\documentclass[12pt,a4paper]{amsart}
\usepackage[utf8]{inputenc}
\usepackage{amssymb,amsmath,amsthm, bbm, mathrsfs}
\usepackage[T1]{fontenc}
\usepackage{graphicx}
\usepackage{comment}
\usepackage{color}
\usepackage{t-angles}
\usepackage[framemethod=tikz]{mdframed}
\definecolor{refkey}{rgb}{0,0,1}
\definecolor{labelkey}{rgb}{1,0,0}
\usepackage{cleveref}
\usepackage[top=2cm,left=3cm,right=2cm,bottom=2cm, nohead]{geometry}
\usepackage{times}
\newtheorem{thm}{Theorem}[section]
\newtheorem{prop}[thm]{Proposition}
\newtheorem{exm}[thm]{Example}
\newtheorem{lem}[thm]{Lemma}
\newtheorem{cor}[thm]{Corollary}
\theoremstyle{definition}
\newtheorem{defn}[thm]{Definition}
\newtheorem{rem}[thm]{Remark}
\numberwithin{equation}{section}

\newcommand{\C}{\mathbb{C}}         
\newcommand{\Z}{\mathbb{Z}}         
\newcommand{\cop}{\Delta}           
\DeclareMathOperator{\id}{id}       

\begin{document}
\title{A note on cohomology for multiplier Hopf algebras}
\author[A.\ Sitarz]{Andrzej Sitarz}\thanks{${}^*$Partially supported by NCN grant 2015/19/B/ST1/03098} 
\address{Institute of Physics, Jagiellonian University,
	prof.\ Stanis\l awa \L ojasiewicza 11, 30-348 Krak\'ow, Poland.\newline\indent
	Institute of Mathematics of the Polish Academy of Sciences,
	\'Sniadeckich 8, 00-950 Warszawa, Poland.}
\email{andrzej.sitarz@uj.edu.pl}   
\author[D.\ Wysocki]{Daniel Wysocki}
\address[b]{Department of Mathematical Methods in Physics, Faculty of Physics, University of Warsaw, Pasteura 5, 02-093 Warszawa, Poland}
\email{Daniel.Wysocki@fuw.edu.pl}   
\subjclass[2010]{58B34, 58B32, 46L87} 
\keywords{multiplier Hopf algebra, Hopf-cyclic, modular pair in involution}
\begin{abstract}
In this note we discuss the possibility of constructing the cosimplicial complex for the multiplier Hopf algebras and extending the cyclicity operator to obtain the Hopf-cyclic cohomology for them. We show that the definition of modular pairs in involution for multiplier Hopf algebras and provide the definition of Hopf-cyclic cohomology 
for algebras of functions over discrete groups.
\end{abstract}
\maketitle
\section{Introduction}
Hopf-cyclic cohomology of Connes and Moscovici \cite{CoMo, CoMo2} has been
the first example of cyclic cohomology with coefficients. Generalized later \cite{HKRS1,HKRS2, JS08} to coefficients valued in certain types of modules 
(stable anti-Yetter-Drinfeld modules) it became an effective tool in the studies
of Hopf algebras. 

In this note we address the problem whether the construction can be extended to
the multiplier Hopf algebras, which are natural generalization of Hopf algebras. We
aim to define the respective cosimplicial objects in the setup of multiplier Hopf
algebras following the definitions of \cite{AvD08, AvD94, DZ99, DvDZ99, KDZ00}.

The note is organized as follows, first we recall basic definitions, then we define the
modular pair in involution for multiplier Hopf algebras and the cosimplicial objects.
Finally, we discuss the cyclicity operator and define the Hopf-cyclic cohomology for
commutative multiplier Hopf algebras of functions over discrete groups. For simplicity we consider algebras over the field of complex numbers (generalization to arbitrary field is in many steps straightforward). Note that in \cite{HMM} a dual point of view has been 
presented, with the cyclic module and cyclic homology in the place of cohomology. 
The difference, however, is that the simplicial object is the usual one, as it involves only
the algebra structure, whereas for the multiplier Hopf algebras, the cosimplicial object
requires more attention as the coproduct is not valued in $H \otimes H$.

\section{Multiplier Hopf algebras.}

Let us recall some basic definitions of multiplier Hopf algebras. A {\em left} multiplier $\alpha$ of  an algebra $A$ is a linear map $\alpha_L: A \to A$ such that $\alpha_L(ab) = \alpha_L(a)b$
for all $a,b \!\in\! A$. Similarly, a {\it right} multiplier $\alpha_R$ satisfies $\alpha_R(ab) = a\alpha_R(b)$. 
A {\em multiplier} is a pair $(\alpha_L, \alpha_R)$ of a left and right multipliers, respectively, such 
that $a \alpha_L(b) = \alpha_R(a)b$ for all $a,b \!\in\! A$. There is a canonical inclusion $\iota$
of $A$ in $M_L(A)$: $\iota(a)(b) = ab$ (and similarly for $M_R(A)$ and $M(A)$).

We define \cite{AvD94} a multiplier Hopf algebra $H$ as a an algebra, 
which is equipped with a comultiplication $\Delta: H \to M(H \otimes H)$ such 
that the following maps (understood as a composition of the elements in the 
multiplier),
\begin{equation}
\begin{split}
&W_R: H \otimes H \ni (a,b) \mapsto \Delta(a)(1 \otimes b) \in H \otimes H,\\
&W_L: H \otimes H \ni (a,b) \mapsto (a \otimes 1) \Delta(b) \in H \otimes H 
\end{split}
\end{equation}
are well defined bijective map, which are coassociative (we refer to \cite{AvD94} 
for details) and $\Delta$ is an algebra morphism in the following sense,
\begin{equation}
W_R(ab,c) = W_R'(a, W_R(b,c)),
\end{equation}
where 
$$ W_R'(a, b \otimes c) = W_R(a,c)(b \otimes 1).$$

We call a multiplier Hopf algebra {\em regular} if $\sigma \circ \Delta$, where 
$\sigma: H \otimes H \to H \otimes H$ is the flip operation, makes $H$ a multiplier 
Hopf algebra again. 

A multiplier  Hopf algebra has the counit $\epsilon\!:\! H \!\to\! \mathbb{C}$, which is
a homomorphism and the antipode $S: H \to M(H)$. If $H$ is regular, then $S$ is 
bijective and thus the image of the antipode is in $H$.

Finally, let us state the {\em extension property} (cf. \cite{AvD94} Proposition A5). 

\begin{lem}\label{exten}
We call a map  $\phi\!:\! A \to M(B)$ non-degenerate if $B$ is generated by $\phi(a)b$ and $b \phi(a)$. 
A non-degenerate homomorphism $\phi\!:\! A \to M(B)$ has a unique extension 
$\tilde \phi\!:\! M(A) \to M(B)$, which is defined by
\begin{equation}
\label{extension}
\tilde \phi(\alpha) x := \phi(\alpha g)h, \qquad \forall \alpha \!\in\! M(A), \forall x \!\in\! B,
\end{equation}
where $x = \phi(g)h$, $g \!\in\! A$, $h \!\in\! B$, by the non-degeneracy of 
$\phi$.
\end{lem}

We shall also use a mode advanced version of the extension property for 
the regular multiplier Hopf algebras that was proven in \cite{AvD08}.

\begin{lem}\label{ext2}
Let $H$ be a regular multiplier Hopf algebra. Then the  maps $\Delta \otimes \hbox{id}$ and $\hbox{id} \otimes \Delta$ defined on $H \otimes H$ 
have natural extension to maps from $M(H \otimes H)$ to $M(H \otimes H \otimes H)$.
\end{lem}

For the rest of the paper we restrict ourselves to regular multiplier Hopf algebras only.

\subsection{Modules over multiplier Hopf algebras}

We define a module, $M$ (left module, right module, bimodule) over a multiplier 
Hopf algebra in the usual sense (see \cite{DvDZ99} for details). 

Each multiplier Hopf algebra is a bimodule over itself with  left and right multiplication. Additionally, since the adjoint action makes sense for the multiplier 
Hopf algebras we have the following left module structure of $H$:
\begin{align}
\hbox{Ad}: a \otimes x  \to a_{(1)} x S(a_{(2)}), \qquad a \in H, 
\end{align}
and makes $H$ again a left $H$-module algebra.

We want to remark, however, that it might be reasonable to consider extended
modules in the sense of van Daele \cite{AvD08}:

\begin{defn}
Let $M$ be a left module over $H$, which is non-degenerate in the sense that if $ax=0$
for $x \in M$ and all $a \in H$ then $x=0$ Consider a space $M'$ of all linear maps 
$\rho: H \ to M$ such that $\rho(ab) = a \rho(b)$, for all $a,b \in H$. Then $M'$ has
a natural left-module structure over $H$ and $x \to \rho_x$, where $\rho_x(a) = ax$ 
is an injective embedding of $m$ in $M'$.
\end{defn}

\subsection{Comodules over multiplier Hopf algebras}

We follow the definitions of \cite{DZ99, KDZ00}. Let $M$ be a vector space. We call $M$ a right comodule 
and $u$ a {\it RR-corepresentation of $H$} if $u$ is injective map $u: M \otimes H \to M \otimes H$ 
that satisfies
\begin{align}
u^{12} u^{13} W_R^{23} = W_R^{23} u^{12}, \label{rr}
\end{align}
where $u^{12}$, for example,  denotes the application of the $u$ map on the 
respective first and second component of the tensor product. Similarly, $M$ is a right comodule
with a {\it RL-corepresentation} if an injective map $v: M \otimes H \to M \otimes H$ satisfies
\begin{align}
v^{12} v^{13} W_L^{23} = W_L^{23} v^{13}. \label{ll} 
\end{align}

For simplicity we shall use name of RR and RL-comodules.
If $u$ ($v$) are bijective then the corepresentation is called {\em regular}.
Finally, we say $M$ is a right comodule if it is RR- and RL-comodule, and
$$ (1 \otimes a) u(m,b) = v(m,a) (1 \otimes b), \quad \forall a,b \in H.  $$  

For regular multiplier Hopf algebras every right comodule has necessarily regular
RR and RL corepresentation and, conversely, every regular RR (RL) corepresentation
gives rise to a right comodule (proposition 2.0 \cite{DZ99}).

\subsection{One-dimensional comodules}

A special and relevant example is given by a $H$-comodule given by a base field, 
$\C$  of the Hopf algebra $H$. Then, the coactions reduce to the maps $H \to H$ 
with certain properties.

\begin{lem}
Let $H$ be a multiplier Hopf algebra. If $(\C, u)$, where $u: H \to H$, 
is a RR-comodule over $H$, then $u$ is a left multiplier of $H$. Analogously, 
if $(\C, v)$ is a $RL$-comodule over $H$, then $v$ is a right multiplier of $H$.
\end{lem}

\begin{proof}
Using (\ref{rr}) and the fact that $\Delta$ is an algebra homomorphism we have for arbitrary 
$a,b,c \in H$:
$$ 
\begin{aligned}
& W_R(u(ab), c) = (u \otimes u) W_R(ab,c) =  (u \otimes u) W_R'(a, W_R(b,c)) \\
& = W_R'(u(a), W_R(b,c)) = W_R(u(a)b, c),
\end{aligned} $$
and since $W_R$ is bijective then
$$ u(ab) = u(a)b, \;\; \forall a,b \in H.$$
Similarly from (\ref{ll}) and coassociativity we have:
$$ 
\begin{aligned}
& W_L(c, v(ab)) =  (v \otimes v) W_L(c, ab) =  (v \otimes v) W_L'(W_L(c,a),b) \\
&= W_L'(W_L(c,a),v(b)) = W_L(c, av(b)),
\end{aligned} 
$$	
and as a consequence 
$$ v(ab) = a v(b), \;\; \forall a,b, \in H.$$
\end{proof}

Hereafter, we will call such a comodule a one-dimensional comodule. Observe that
a generalization to an arbitrary field is straightforward.

\begin{cor}
Let $H$ be a multiplier Hopf algebra over $\C$. The one-dimensional 
right comodule over $H$ is given by a multiplier $u$ of $H$.
\end{cor}

Notice that for regular multiplier Hopf algebras the extension property (\ref{extension}) allows us to extend the coproduct  
to $\Delta\!:\! M(H) \to M(H \otimes H)$ (respectively one can use one-sided multipliers instead). Thus, in case of a one-dimensional 
right-right (right-left) comodule we have the following property.

\begin{lem}
A one-dimensional right-right (left-right) comodule over a MHA is determined
by a group like left (right) multiplier $u (v)$:
$$ \Delta u = u \otimes u. $$
\end{lem}
\begin{proof}
Since the maps $W_R$ and $W_L$ are bijective then $H \otimes H$ is
spanned by $\Delta(a) (b \otimes c)$ and $(a \otimes b) \Delta(c)$. Therefore
from the extension property we see that $\Delta$ extends as a linear map from $L(H)$ ($R(H)$) to $L(H \otimes H)$ ($R(H \otimes H)$). As in both cases the tensor
product of the multipliers is included in the multiplier of the tensor product we
see that the expressions like $\Delta(u) = u \otimes u$ makes sense. Let us
compute it (for the left multiplier alone). First of all, by definition:
$$ W_R(ua,b) = \Delta(ua)(1 \otimes b) = \Delta(u) W_R(a,b).$$ 
Rewriting the condition (\ref{rr}) we have:
$$ W_R(ua,b) = (u \otimes u) W_R(a,b). $$
Therefore
$$ \Delta(u) = u \otimes u,$$
in the sense of the unique extension of $\Delta$ map to the (left) multiplier
\end{proof}

\begin{rem}
The above construction extends nicely to the finite-dimensional case. Any finite-dimensional right comodule over a multiplier Hopf algebra (of 
dimension $N$) is a matrix of multipliers $u_{ij}$, $i,j =1,\ldots,N$ which 
satisfies:
$$ \Delta(u_{ij}) = \sum_{k=1}^N u_{ik} \otimes u_{kj}. $$
\end{rem}

\begin{exm}
Consider an algebra of complex valued functions with finite support over 
a discrete group $G$, which is a typical simplest example of a regular 
commutative multiplier Hopf algebra. We have for the generating functions 
$e_p, e_h$, $p,h \in G$:
$$ W_R(e_p, e_h)  = e_{ph^{-1}} \otimes e_h, \;\;\;\; 
   W_L(e_p, e_h)  = e_{p} \otimes e_{p^{-1}h}.
$$   
If we set the multiplier $u$ by defining it as:
$$ u(e_g) = f(g) e_g. $$
then we see that the condition that $u$ defines a coassociative right 
comodule becomes:
$$ f(gh) = f(g) f(h). $$
Observe that since the algebra of functions is commutative then we
necessarily have for a nonzero functions:
$$ f(ghg^{-1}h^{-1}) = f(e) = 1, $$
where $e$ is the neutral element of $G$.
A typical case is $G= \mathbb{Z}$ where the group-like multipliers
are given by exponential function:
$$ u_\alpha (e_m) = e^{\alpha m} e_m, $$
which satisfies the equation (\ref{rr}).
\end{exm}

\section{Modular pair in involution for multiplier Hopf algebras}

Let us recall that a modular pair in involution \cite{CoMo} for a Hopf algebra $H$ is a pair 
$(\delta, \sigma)$, where $\delta: H \to \C$ is a character on $H$ and $\sigma \in H$ 
is a group-like element, i.e. $\Delta(\sigma) = \sigma \otimes \sigma$, satisfying $\delta(\sigma) = 1$, such that $S^2_{(\delta,\sigma)}(h) = \sigma h \sigma^{-1}$ for every $h \in H$, where 
$S_{(\delta,\sigma)}(h) := \delta(h_{(1)}) S(h_{(2)})$ is the twisted antipode.

For a multiplier Hopf algebra we propose the following definition.
\begin{defn}
Let $H$ be a regular multiplier Hopf algebra. We say that $(\delta,\sigma)$ is 
a modular pair in involution if $\delta$ is a character of $H$, $\sigma \in M(H)$ 
is group-like, so that $\Delta(\sigma) = \sigma \otimes \sigma$, $\delta(\sigma)=1$ 
and the map $S_{(\delta, \sigma)}$ is defined on $H$ in the following way,
\begin{equation}
  \delta(a) S_{(\delta, \sigma)}(h) := (\delta \otimes S)W_L(a, h), \;\; a,h\in H,
  \label{twiant}
 \end{equation} 
and satisfies
\begin{equation}
S^2_{(\delta,\sigma)}(h) = \sigma h \sigma^{-1}. \label{twiant2}
\end{equation} 

Note, that $S_{(\delta, \sigma)}$ does not depend on $a$ and $\sigma^{-1}$ is 
a unique element of the multiplier, in fact:
$$ \sigma^{-1} = S (\sigma). $$
\end{defn}

First of all we need to show that the definition is self-consistent. Due to the
extension property $\delta$ extends to $M(H)$ so we can require that 
$\delta(\sigma)$ is $1$. Further, as the multiplier Hopf algebra is assumed to be
regular we know that by construction $S_{(\delta, \sigma)}(h) \in H$. It remain
only to verify that the definition of $S_{(\delta, \sigma)}(h)$ does not depend 
on $a$. 

We use here the property of regular multiplier Hopf algebras, which guarantees that for every finite set of elements $a_i$ there exists a common local unit (left and right), that is an element $e$ such that $a_i e = a_i$ (respectively, $a_i = e a_i$) for each $i$. 
(see \cite{DvDZ99} Proposition 2.2).  Using this and the identity:
$$ W_L(ae,b) = (a \otimes 1) W_L(e,b), $$ 
we immediately have that for any two different $a_1$ and $a_2$ (such that $\delta(a_1) \delta(a_2) \not= 0$) choosing a suitable $e$ we have the right-hand  side of (\ref{twiant}):
$$ 
\begin{aligned}
S_{(\delta, \sigma)}(h) &= \frac{1}{\delta(a_1)} (\delta \otimes S)W_L(a_1, h) = (\delta \otimes S)(a_1 \otimes 1) W_L(e, h) \\
&= \frac{\delta(a_1)}{\delta(a_1)} (\delta \otimes S) W_L(e, h) =   \frac{\delta(a_2)}{\delta(a_2)} (\delta \otimes S) W_L(e, h) \\
&= \frac{1}{\delta(a_2)}  (\delta \otimes S)(a_2 \otimes 1) W_L(e, h)  =\frac{1}{\delta(a_2)}  (\delta \otimes S)W_L(a_2, h). 
\end{aligned}
$$
\begin{exm}
Consider again a canonical example of the algebra of functions with finite 
support on a discrete  group $G$. As the algebra is commutative a character 
of the algebra is given by an element of the group $g$, $\delta_g$. A group-like multiplier $\sigma$ (which we have studied in the previous example) is defined by a multiplicative morphism $f$ from $G$ to the field, here additionally we need to add the condition that $f(g) = 1$. 

Using the basis of the generating functions, we have:
$$ \delta_g(e_h) = \delta_{gh}, \qquad \sigma e_h = f(h) e_h. $$
The twisted antipode becomes:
$$ S_{(\delta,\sigma)} e_h =  (e_{h^{-1}g}). $$

Let us note, however, that the twisted antipode will satisfy the condition
$$ S_{(\delta,\sigma)}^2 = \hbox{id}, $$
only if $G$ is abelian or if $g=e$, in the latter case $\delta$ is the counit and
the twisted antipode is just $S$. 
\end{exm}	

\begin{exm}
We shall consider a genuine noncommutative and noncocommutative example 
based on the multiplier Hopf algebra acting on the double noncommutative torus \cite{HM}. The explicit description through the generators and relations and $W_L, W_R$ maps was provided in \cite{SiEq}, were we provide a slightly different approach and 
start already with the multiplier algebra and the extension of the coproduct map. Let $A=C(\Z^2)$, and $G=\Z_2$ with generator $x$, $x^2=1$. Consider the algebra 
$A \otimes \C G$, however, with a slightly modified product between $A$ and $x$ 
and  coproduct of $x$:
\begin{equation}
f x = x \hat{f}, \qquad \cop(x) = e^{i\theta(i,j)} x \otimes x,
\end{equation}
where
$$ \hat{f}(i_1,i_2) = f(i_2, i_1)$$
and $\theta$ is a cocycle on $\Z^2$
$$
\theta(i_1,i_2,j_1,j_2) = i_1 j_2 - i_2 j_1,
$$
with the coproduct on $A$ arising from the abelian group structure of $\Z^2$ and
the usual product on $A$ and $\C{}G$.

The above coalgebra is an example of a regular multiplier Hopf algebra, which, as 
an algebra is, in fact, a crossed product of the algebra of functions of the discrete 
group $\Z^2$ by the group $\Z_2$. 

First of all, we shall determine a group-like multiplier. If $\sigma = f +  hx$ is a group-like
element it shall satisfy, $f,h \in C(\Z^2)$,
$$ \Delta(f) + \Delta(h) \Delta(x) = (f +  hx) \otimes (f +  gx).$$
and this can be true only if $h=0$ and $\sigma$ is group-like for $C(\Z^2)$, which
means that $\sigma(i_1, i_2) = e^{\alpha_1 i_1 + \alpha_2 i_2}$.

A character of the algebra, $\delta$ must satisfy $\delta(x) = \pm 1$ and 
$\delta(f) = f(i,i)$ for some $(i,i) \in \Z^2$. The associated twisted antipode
is:
$$ 
(S_{(\delta,\sigma)} f)(j_1,j_2)= f(i-j_1,i-j_2), \qquad (S_{(\delta,\sigma)} x) = \pm x.
$$
One can easily check that only the character $\delta(f) = f(0,0)$ is possible for
a modular pair in involution, and $\sigma, \delta$ is a modular pair in involution iff $\sigma$ satisfies $\sigma(i,j)=\sigma(j,i)$. 
\end{exm}
\section{Cosimplicial modules for multiplier Hopf algebras.}

Let us recall here the core definition of a cosimplicial module for a Hopf algebra $H$.
Note that this uses, of course, only the coproduct (for the pre-cosimplicial structure,
coface maps) and the counit (to define codegeneracy maps). Although our motivation for the choice follows from the module over cyclic category for Hopf algebras, 
as proposed by Connes-Moscovici \cite{CoMo}, yet we use only the cosimplicial
part in this section.

\begin{lem}[\cite{CoMo}]
Let $H$ be a Hopf algebra with a modular pair in involution. Then, with
$E^n = H^{\otimes n}$, and the following linear maps, 
\begin{equation}
\begin{aligned}
&\delta_i: H^{\otimes n} \to H^{\otimes n+1}, \qquad
&\sigma_i: H^{\otimes n} \to H^{\otimes n-1}, \qquad
\end{aligned}
\end{equation}
defined as
\begin{equation}
\begin{aligned}
&\delta_i(h^1 \otimes \ldots \otimes h^n) := h^1 \otimes \ldots \otimes h^{i-1} \otimes \Delta(h^i) \otimes h^{i+1} \otimes \ldots \otimes h^n,\\
&\delta_0(h^1 \otimes \ldots \otimes h^n) := 1 \otimes h^1 \otimes \ldots \otimes \ldots \otimes h^n,\\
&\delta_{n+1}(h^1 \otimes \ldots \otimes h^n) :=  h^1 \otimes \ldots \otimes \ldots \otimes h^n \otimes \sigma,\\
&\sigma_i(h^1 \otimes \ldots \otimes h^n) := 
\epsilon(h^i) h^1 \otimes \ldots \otimes h^{i-1} \otimes h^{i+1} \otimes \ldots \otimes h^n,\\
&\tau_n(h^1 \otimes \ldots \otimes h^n) := \left(\Delta^{n-1} S_{(\delta,\sigma)}(h^1) \right)
\left( h^2 \otimes \ldots \otimes h^{n} \otimes \sigma \right),
\end{aligned}
\label{cosimp}
\end{equation}
$\{ E^n ,\delta_i, \sigma_j, 0\leq i \leq n\!+\!1, 0<j<n\!+\!1,\}_{n\geq0}$ is a cosimplicial 
module. 
\end{lem}

Of course, for the multiplier Hopf algebras this does not work directly as the maps
lead out of the space $H^{\otimes n}$. However, using the standard tools we might 
be able to extend them and propose two possible versions for the multiplier Hopf
algebras.

Let us denote by $M^n(H)$ the multiplier algebra of $M(H^{\otimes n})$ and 
by $N(H)$ the algebra spanned by $H,1,\sigma$, which can be understood
as a subalgebra of $M(H)$.

\begin{defn}\label{m0def}
Let us define  by $M_0^n(H)$ a subspace of $M^n(H)$, which consists 
of all elements $z$ in the multiplier algebra of $H^{\otimes n}$, 
such that for every $w_i = a_1 \otimes  a_2 \otimes \cdots 1_i \cdots \otimes a_{n}$,
$i=1,2,\ldots n$ we have,
$$ z w_i \in H \otimes H \otimes \cdots N_i(H) \cdots \otimes H \ni w_i z, $$
where $N_i(H)$ is on the $i$-th place.
\end{defn}

\begin{rem}
Observe that for regular multiplier Hopf algebras $\Delta(H) \subset M_0^2(H)$
and that $M_0^n(H)$, in fact, is an algebra.
\end{rem}
We have:
\begin{prop}
Each of the maps $\delta_i$, extends as a map between 
$M_0^n(H)$ and $M_0^{n+1}(H)$, similarly, maps $\sigma^i$ extend to maps 
between $M_0^{n+1}(H)$ and $M_0^n(H)$.
\end{prop}
\begin{proof}
First, observe that the first part of the statement is trivial for the maps $\delta_0$ 
and $\delta_{n+1}$ as they map, respectively, $M_0^n$ to $N(H) \otimes M_0^n(H) \subset 
M_0^{n+1}(H)$  (and $M_0^n(H) \otimes N(H) \subset M_0^{n+1}(H)$, respectively). 
The inclusions are obvious.

To see that maps $\delta_i$ extend we repeat the arguing that is used in \cite{AvD08},
Proposition 1.10. Let $y \in M^n(H)$, which means that for any $a_1, a_2,\ldots, a_n \in H$, we have:
$$
y (a_1 \otimes a_2 \otimes \cdots \otimes a_n) \in H^{\otimes n} \ni (a_1 \otimes a_2 \otimes \cdots \otimes a_n)y. 
$$

For a fixed $y$ let us define an element $z \in M_L^{n+1}(H)$
in the following way:
\begin{equation}
\begin{aligned}
z (a_0 \otimes a_1 \otimes \cdots \otimes a_n) = 
\sum_i  \delta_1   \left( y (r_i \otimes a_2  \otimes \cdots \otimes a_n) \right)
 (s_i \otimes 1 \otimes \cdots \otimes 1 ) ,
\end{aligned}
\label{defdelta}
\end{equation} 
where we use the fact that for a regular multiplier Hopf algebra there exists 
$r_i,s_i \!\in\! H$  such that 
$$  a_0 \otimes a_1 
= \sum_j  \Delta(r_j) (s_j \otimes 1). $$

Then, the same arguments as in the above mentioned proposition ensure that $z$
is a well defined element of $M_L^{n+1}(H)$, so setting $z=\delta_1(y)$ shows 
that $\delta_1$ has an extension as a map $M_L^n(H) \to M_L^{n+1}(H)$. 
Next we need to demonstrate that it maps $M_0^n(H)$ to $M_0^{n+1}(H)$. 
To prove it we need to consider three cases. First, if one of the elements 
$a_2,\ldots, a_n$ is equal to $1$ we use the assumption that $y \in M_0^n(H)$ 
and as a consequence the argument of $\delta_1$ on the right-hand side of (\ref{defdelta})
is in tensor product $H \otimes H \otimes \cdots N(H) \cdots \otimes H$, where a single
$N(H)$ is in the same place as $1$. Since $\delta_1$ acts as $\Delta$ on the 
first element of the tensor product we see that the right-hand side is again in 
the same target space.  If $a_0=1$ then we check that 
$$ z (1 \otimes a_1 \otimes \cdots \otimes a_n) = 
\sum_i  \delta_1   \left( y (1\otimes a_2  \otimes \cdots \otimes a_n) \right)
(1 \otimes a_1 \otimes \cdots \otimes 1 ),$$
is in $N(H) \otimes H \cdots \otimes H$. Indeed by definition, the first element
of the tensor product in the argument of $\delta_1$ on the right-hand side is
in $N(H)$, and then we know that for any $x \!\in\!N(H)$ we have
$$ \Delta(x) (1 \otimes a) \in N(H) \otimes H, \;\;\; \Delta(x) (a \otimes 1) \in H \otimes N(H),$$
which is sufficient to show the desired result. Similarly, if $a_1=1$, we take 
(\ref{defdelta}) with $r\!=\!1$,$s\!=\!a_0$ and use the same argument as above.
To extend the maps $\sigma_i$ we use an analogous construction. Let us take
$y \!\in\! M_0^n(H)$. We define $z\!=\!\sigma_i(y) \!\in\! M_L^{n-1}(H)$ in the following 
way, for all  $a_1, a_2,\ldots, a_n \in H$, we put:
$$ 
\begin{aligned}
\epsilon(a_i) z (a_1 \otimes a_{i-1} \otimes a_{i+1} \cdots \otimes a_n) = 
\sigma_i \left( y (a_1 \otimes \cdots \otimes a_n)\right).
\end{aligned}
$$ 

The above proof demonstrates that $\delta_i(y)$ and $\sigma_j(y)$ are well-defined
left multipliers, however, repeating analogous arguments we can show that they
are also right-multipliers obeying also the second identity from definition \ref{m0def}
and hence they are indeed in the respective $M_0^*(H)$ modules.
\end{proof}

\begin{rem}
The above construction uses only the extension of the coproduct to the multiplier 
algebra and the coassociativity as well as compatibility of the compatibility of the 
counit with the coproduct for the multiplier Hopf algebras. Note that if the arguments
of $\delta_i, \sigma_j$ are in $M^n(H)$, the definition still holds (though of course
the value is only in the respective multiplier and not in its restricted version).
\end{rem}
Summarizing we have,
\begin{prop}\label{cosimp-pro}
With the maps $\delta_i$,$\sigma_j$ defined as before, 
$\{M_0^n(H)\}_{n\geq0}^\infty$ as well as $\{M^n(H)\}_{n\geq0}$ are cosimplicial 
modules.
\end{prop}
\begin{proof}
It remains to prove the rules for the composition of maps. First let us check the composition of maps $\delta_i$ satisfies 
	$\delta_i \delta_j = \delta_{j+1} \delta_i$, $i \leq j$. We skip the
	trivial case when $i=0$ or $j=n+1$ (then it is straightforward) and 
	concentrate on the nontrivial case $0 < i \leq j < n+1$. First, we consider $i<j$,
	for simplicity fixing $i=1, j=2$ (all other cases will be analogous). 
	Let us write the element in the left multiplier, $z$ defined through in 
	(\ref{defdelta}), for the product of $\delta_1,\delta_2$ acting on $y \in M_0^n(H)$,
	\begin{equation}
	\begin{aligned}
	& \quad  \delta_1(\delta_2(y)) (a_0 \otimes a_1 \otimes \cdots \otimes a_n \otimes a_{n+1}) = \\
	&= \sum_{i}  \delta_1 \left( \delta_2(y) (r_i \otimes a_2 \otimes a_3  \otimes \cdots \otimes a_{n+1}) \right)
	( s_i \otimes 1 \otimes \cdots \otimes 1 ) \\
	&= \sum_{i,j}  \delta_1 \left( \delta_2 \left( y ( r_i \otimes p_j  \otimes \cdots \otimes a_n) \right)
	(1  \otimes  1 \otimes q_j  \otimes \cdots \otimes 1 ) \right) (s_i \otimes 1 \otimes \cdots \otimes 1 ).
	\end{aligned}
	\label{coadelta1}
	\end{equation} 
	where we use 
	$$  a_0 \otimes a_1 
	= \sum_j  \Delta(r_j) (s_j \otimes 1), \qquad
	a_2 \otimes a_3 
	= \sum_j  \Delta(p_j)(q_j \otimes 1). 
	$$
	On the other hand 
	\begin{equation}
	\begin{aligned}
	& \quad \delta_3(\delta_1(y)) (a_0 \otimes a_1 \otimes \cdots \otimes a_n \otimes a_{n+1}) = \\
	&= \sum_{j}  \delta_3 \left( \delta_1(y) (a_0 \otimes a_1 \otimes p_j  \otimes \cdots \otimes a_{n+1}) \right)
	(1 \otimes 1 \otimes q_j \otimes \cdots \otimes 1 ) \\
	&= \sum_{i,j}  \delta_3 \left( \delta_1 \left( y ( r_i \otimes p_j  \otimes \cdots \otimes a_n) \right)
	(s_i \otimes 1 \otimes \cdots \otimes 1 ) \right) (1 \otimes  1 \otimes q_j \otimes  \cdots \otimes 1 ).
	\end{aligned}
	\label{coadelta1}
	\end{equation} 
	To see that both expressions are identical it is sufficient to use the fact that for the multiplier
	Hopf algebras we have:
	$$ \left( (\id \otimes \Delta) (a_1 \otimes a_2) \right) (h \otimes 1 \otimes 1)
	 = (\id \otimes \Delta) (a_1 h \otimes a_2), $$
	where the equality makes sense in the respective multiplier.Similar arguments, which are directly based on the coassociativity of the coaction for the multiplier Hopf algebras can be applied in the case $i=j$. 
		
	Finally observe that the relations between the coface maps and codegeneracy operators again follow directly from the properties of the counit extended to the respective multiplier.
\end{proof}	
\begin{defn}
	We define a full Hochschild Hopf-cohomology of a multiplier Hopf algebra with
	respect to modular element $\sigma$ as:
	$$ HH_{\sigma}^n(H) := {\hbox{ker}\, b}_{M^{n}(H)} / { \hbox{Im} \,b }_{M^{n-1}(H)}, $$
	where $b = \delta_0 - \delta_1 + \cdots + (-1)^{n+1} \delta_{n+1}$.
	
	Since we know that coface maps and in consequence, the coboundary $b$ restricts
	to the restricted multiplier, we can equally define the minimal Hochschild Hopf-cohomology of a multiplier Hopf algebra,with respect to modular 
	element $\sigma$, as:
	
	$$ HH_{\sigma,0}^n(H) := {\hbox{ker}\, b}_{M_0^n(H)^{n}} / { \hbox{Im} \,b }_{M_0^{n-1}(H)}.$$
\end{defn}
Observe that out of the modular pair it is only $\sigma$ that enters the definition of the coboundary. Although the cochains start with $n\!=\!0$, with $M^0(H)=\C$,  the first nontrivial cohomology group is $HH^1_\sigma$. Indeed, the coboundary $b$ acting
on $c\!\in\!\C$ gives $bc = c(1-\sigma) \in M_0(H) \!\subset\! M(H)$, and its kernel is 
trivial (unless $\sigma\!=\!1$). Before we proceed with the further restrictions of the module, let us look at the motivating example.
\subsection{The discrete group $G$}

Let $H\!=\!C_0(G)$ be an algebra of functions with finite support over a discrete group with the standard basis $e_g$ and let us fix a multiplicative morphism $\sigma: G \to \mathbb{C}$. 
The multiplier of $M^n(C_0(G))$ is a space of all functions over $G^{\times n}$ whereas 
$M_0^n(C_0(G))$ is the space of functions such that when evaluated on $n\!-\!1$ 
points give a linear combination of a function with finite support, identity and 
$\sigma$ in the remaining variable.

\begin{lem}
Taking a function $F \in M^n(C_0(G))$ we have:
\begin{equation}
\begin{aligned}
b F (g_1, g_2, \ldots, g_{n+1}) =& F(g_2, \ldots, g_{n+1}) + (-1)^i \sum_i F(g_1, \ldots g_i g_{i+1}, \ldots, g_{n+1}) \\ 
&+ (-1)^{n+1} F (g_1, g_2, \ldots, g_{n}) \sigma(g_{n+1}),
\label{bdisg}
\end{aligned}
\end{equation}	
and we immediately see that it also maps elements of $M^n(C_0(G))$ to $M^{n+1}(C_0(G))$.
\end{lem}
{
\begin{proof}
Any function $F$ of $n$ variables over a discrete group $G$ can be understood
as an element of the multiplier $M((C_0(G))^{\otimes n})$ and thus, the evaluation of the multiplier on elements of $(C_0(G))^{\otimes n}$ corresponds to pointwise multiplication, i.e.
$$ 
\left[ F \bigl( e_{g_1} \otimes e_{g_2} \otimes \cdots \otimes e_{g_n} \bigr)\right](g_1 \otimes \ldots \otimes g_n) = 
F({g_1, g_2, \ldots, g_n} ),
$$
where each $e_h$ are the basis functions over $G$. We can then compute the explicit actions of $\delta_i$ following the definitions from Proposition 4.4 and (\ref{defdelta}). 
For example, 
\begin{equation*}
(\delta_1 F)\bigl( e_{g_1} \otimes e_{g_2} \otimes \cdots \otimes e_{g_n} \otimes e_{g_{n+1}} \bigr) = \delta_1 \bigl( F\bigl( e_{g_1 g_2} \otimes e_{g_2} \otimes \cdots \otimes e_{g_{n+1}} \bigr) \bigr)
\bigl( e_{g_1} \otimes 1 \cdots \otimes 1 \bigr),
\end{equation*}
where we have used
$$ e_{g_1} \otimes e_{g_2} = \Delta(e_{g_1 g_2})(e_{g_1} \otimes 1). $$
Evaluating this expression on $g_1 \otimes \ldots \otimes g_{n+1}$, we get
\begin{equation*}
\begin{split}
(\delta_1 F)(g_1 \otimes \ldots \otimes g_{n+1}) &=  \delta_1 \bigl( F\bigl( e_{g_1 g_2} \otimes e_{g_2} \otimes \cdots \otimes e_{g_{n+1}} \bigr) \bigr)(g_1 \otimes \ldots \otimes g_{n+1})\\
&= \bigl( F\bigl( e_{g_1 g_2} \otimes e_{g_2} \otimes \cdots \otimes e_{g_{n+1}} \bigr) \bigr) (g_1 g_2 \otimes g_3 \otimes \ldots \otimes g_{n+1})\\
&= F(g_1 g_2 \otimes g_3 \otimes \ldots \otimes g_{n+1}).
\end{split}
\end{equation*}
Computing in a similar way the action of other $\delta$ maps we obtain the formula
(\ref{bdisg}).
\end{proof}
}
As we can easily see we have,
\begin{prop}
The Hochschild Hopf-cohomology of $C_0(G)$ is equal to the cohomology of
group $G$ with values in $\C$, with the module structure of $\C$ set by $\sigma^{-1}$.  
\end{prop}	
\begin{proof}
The definition of the group cohomology uses cochains complex, with $n$-cochains
defined as $G$-module valued functions and the coboundary,
$$ 
\begin{aligned}
d \phi(g_1,g_2,&\ldots,g_{n+1}) = g_1 \phi(g_2,\ldots,g_{n+1}) - \phi(g_1g_2,\ldots,g_{n+1}) + \cdots \\
& +(-1)^{i} \phi(g_1,\ldots, g_i g_{i+1}, \ldots, g_n) + \cdots + (-1)^{n+1}  \phi(g_1, g_2,\ldots,g_{n}).
\end{aligned}
$$
It is easy to see that the map $\Xi$:
$$ \Xi(F)(g_1,\ldots,g_n) = F(g_n^{-1}, \ldots,g_1^{-1}), $$
is a morphism of cochain complexes $(C(G^n),\sigma,b)$ and 
$(C(G^n), \sigma^{-1}, d)$.
\end{proof}

The restriction of the cochains to the subspace denoted $M_0^n(C_0(G))$ is interesting
from the point of view of restrictions of cohomology. The usually considered restriction 
is to the bounded functions yet the above construction yields a different version.
We shall illustrate it with an example of $G=\Z$.

\begin{exm}
Let us consider $G=\Z$ and the first cohomology group $HH_\sigma^1(C_0(\Z))$. 
The group-like element in the
multiplier is an exponential function $\sigma(n) = e^{\alpha n}$. First, the $0$-th 
cochains are identified with $\C$ itself and their image under map $b$ are functions 
of the type:
$$ f(n) = \beta(e^{\alpha n}-1). $$
The condition that a function $F\!:\!G \to \C$ is in the kernel of $b$ reads,
$$F(n+m) = F(m) + F(n) e^{\alpha m} , $$
and, as this is an easy recurrence relation, it can be explicitly solved to give exactly
$$ F(n) = \beta (e^{\alpha n} -1), $$
which is the image  of $b$, hence we conclude that $HH^1_\sigma(C_0(\Z))\!=\!0$. 
Observe that the above function is also in the restricted cochain complex (as
it is a linear combination of a constant function and $\sigma$) so, we also 
have $ HH^1_{\sigma,0}(C_0(\Z)) = 0$.
\end{exm}

\section{The Connes-Moscovici Hopf-cyclic cohomology 
	           for functions over discrete groups.} 

The Hopf-cyclic cohomology of Hopf algebra has been constructed by 
Connes and Moscovici \cite{CoMo} on the basis of the cosimplicial module
\ref{cosimp} using the nontrivial cyclicity operator $\tau_n$:
\begin{equation}
\tau_n(h^1 \otimes \ldots \otimes h^n) := \left(\Delta^{n-1} S_{(\delta,\sigma)}(h^1) \right)
\left( h^2 \otimes \ldots \otimes h^{n} \otimes \sigma \right).
\label{taun}
\end{equation}

First of all, observe that both the maps $S_{\delta,\sigma}$ (acting is does in (\ref{taun})) as well as the coproduct do extend to the multiplier $M^n(H)$. 
The problem, however, is with the extension of the action of the resulting 
tensor product (in the Hopf algebra case) to the multiplier. In other words, 
the problem is to generalize the multiplication map $\mu: a \otimes b \to ab$
to $M(H \otimes H) \to M(H)$. We leave the question, whether this problem can 
be circumvented to future work, and concentrate here on the easy case when
this is possible, namely on commutative regular multiplier Hopf algebras.

\begin{lem}\label{lemtau}
Let $H$ be a commutative regular multiplier Hopf algebra. Then for any elements of
the algebra, $a_0, a_1,\ldots,a_n \!\in\! H$ and any $y \in M^n(H)$ the definition,
\begin{equation} \label{tauc}
\tau_n(y) 
\bigl( \Delta^{n-1}(S_{\delta,\sigma}a_0) 
(a_1 \otimes \cdots \otimes a_n) \bigr)
= \tau_n \bigl( y (a_0 \otimes a_1 \cdots \otimes a_{n-1}) \bigr)  (1 \otimes \cdots \otimes 1 \otimes a_n ),
\end{equation}
gives a well defined map from $M^n(H)$ into itself.
\end{lem}
Observe that the map $\tau_n$ in (\ref{tauc}) is defined on an element of the multiplier
$y$ in principle gives a left multiplier only, however, as $H$ is commutative it is equal
to the right multiplier. Moreover, the arguments of $\tau_n(y)$ are elements of the
tensor product $H^{\otimes n}$ that are of very special form, however, the regularity
of $H$ will ensure that this plays no role and the definition is valid.

As the typical case of a commutative multiplier Hopf algebra is that of $H=C(G)$,
where $G$ is a discrete group, we shall omit the abstract proof of lemma \ref{lemtau} 
and provide an explicit formula for $\tau_n$ in that case. Fixing the notations as before with the modular element $\sigma$ and a character $\delta$ (which we choose, motivated
by Example 3.2 to be the counit, $\delta=\varepsilon$), we have.
\begin{prop}\label{taudis}
Let $F: G^n \to \C$ be an element of the multiplier $M^n(C_0(G))$. Then
the cyclicity operator $\tau_n$ acts in the following way:
\begin{equation}
\tau_n(F)(g_1,g_2,\ldots,g_n) = F\bigl((g_1 g_2 \cdots g_n)^{-1}, g_1,\ldots, g_{n-1}\bigr)
\sigma(g_n).
\label{cyclg}
\end{equation}
and $\tau_n$ satisfies the same identities as the cyclicity operator for the 
cosimplicial module, that is
$$
\begin{aligned}
&\delta_{p-1} \tau_n = \tau_{n+1} \delta_p, \quad \forall p = \{1, \ldots, n\}; \qquad
& \tau_{n+1} \delta_0 = \delta_n \\
& \sigma_{p-1} \tau_n = \tau_{n+1} \sigma_p, \quad \forall p = \{1, \ldots, n\};\qquad
& \tau_n \sigma_0 = \sigma_n \tau_{n+1}^2, \qquad
& \tau_n^{n+1} = id,
\end{aligned}
$$
\end{prop}
\begin{proof}
{ 
The formula (\ref{cyclg}) is a straightforward implementation of (\ref{tauc}). Let us
compute it explicitly, using the following identity:
$$ e_{g_1} \otimes e_{g_2} \cdots e_{g_n} 
= \bigl( \Delta^{n-1} S e_{(g_1 g_2 \cdots g_n)^{-1}} \bigr)  \bigl( e_{g_1} \otimes
 e_{g_2} \cdots \otimes e_{g_n} \bigr).
$$
Then,
\begin{equation*}
(\tau_n F)\bigl( e_{g_1} \otimes e_{g_2} \otimes \cdots \otimes e_{g_n}  \bigr) = \tau_n \bigl( F(e_{(g_1 g_2 \cdots g_n)^{-1}} \otimes e_{g_1} \otimes \cdots \otimes e_{g_{n-1}}) \bigr) 
(1 \otimes 1 \otimes \cdots \otimes e_{g_n}).
 \end{equation*}
Evaluating this expression on $g_1 \otimes \ldots \otimes g_n$, we get
\begin{equation*}
\begin{split}
(\tau_n F)&(g_1 \otimes \ldots \otimes g_n) \\
&=  \tau_n \bigl( F(e_{(g_1 g_2 \cdots g_n)^{-1}} \otimes e_{g_1} \otimes \cdots \otimes e_{g_{n-1}}) \bigr)(g_1 \otimes \ldots \otimes g_n)\\
&= \bigl( F(e_{(g_1 g_2 \cdots g_n)^{-1}} \otimes e_{g_1} \otimes \cdots \otimes e_{g_{n-1}}) \bigr)((g_1 \ldots g_n)^{-1} \otimes g_1 \otimes \ldots \otimes g_{n-1}) \sigma(g_n)\\
&= F((g_1 \ldots g_n)^{-1}, g_1, \ldots, g_{n-1}) \sigma(g_n).
\end{split}
\end{equation*}
}
 
The only
nontrivial identity of the relations above is the cyclicity of $\tau_n$, which we
prove explicitly by direct computation,
$$
\begin{aligned}
\bigl( (\tau_n)^k & F \bigr)(g_1,\ldots,g_{n}) = 
\bigl( (\tau_n)^{k-1} F \bigr)(\bigl((g_1 \cdots g_{n})^{-1}, g_1,\ldots, g_{n-1}\bigr)
\sigma(g_{n}) \\
&= 	\bigl( (\tau_n)^{k-2} F \bigr)(\bigl(g_n, (g_1 \cdots g_{n})^{-1}, g_2,\ldots, g_{n-2}\bigr)
\sigma(g_{n-1}) \sigma(g_{n}) \\
& = \ldots \\
& = F(g_{n-k+2}, g_{g-k+3}, \ldots, (g_1 \cdots g_{n})^{-1}, \ldots, g_{n-k})
\sigma(g_{n-k+1}) \cdots \sigma(g_{n}),
\end{aligned}
$$
for $k=1,\ldots n-1$. For $k=n$ we have:
$$ \bigl( (\tau_n)^n F \bigr)(g_1,\ldots,g_{n}) = 
 F(g_{2}, g_{3}, \ldots, g_n, (g_1 \cdots g_{n})^{-1})
\sigma(g_{1}) \cdots \sigma(g_{n}),
$$
and it is easy to see that $(\tau_n)^{n+1} = \hbox{id}$.
\end{proof}

As a consequence we may restrict the coboundary operator to the subcomplex
of cyclic cochains in the cosimplicial complex. Therefore we obtain,
\begin{lem}
Let $\{M(C(G))^n, \delta_i, \sigma_j\}$ be the cosimplicial module of  Lemma 4.1, then with the above defined $\tau_n$ it becomes a cocyclic module and we can restrict the coboundary map to $M(C(G)^n_\tau$, which are cochains $z$ that satisfy
$(-1)^n \tau_n(x) = x$.  
\end{lem}

The resulting Hopf-cyclic cohomology of $C(G)$ is an interesting object which 
leave for future study. Let us finish this section with an important remark.

\begin{rem}
If $F$ is in $M_0^n(C(G))$, that is evaluated on arbitrary $n\!-\!1$ arguments,
it is a finite support function of the remaining argument then $\tau_n(F)$ is
not necessarily in $M_0^n(C(G))$.

To see the counterexample take $G=\Z$ and $F(m,n) = q(n) \delta_{-m,n}$ for
any function $q$. It certainly satisfies the assumptions, yet as we compute
$$ \tau_2(F)(m,n) = F(-m-n,m) \sigma(n) = q(-m-n) \delta_{m+n,m} \sigma(n),$$
we see that at  $m\!=\!0$ it is a function $q(-n) \sigma(0) = q(-n)$,  which is not 
finitely supported and not necessarily in the algebra generated by $C_0(G)$ and $\sigma$. 
\end{rem}

The above observation is very significant, as it demonstrates that the cyclicity
operator $\tau$ cannot be restricted to the minimal cosimplicial complex that we studied in the previous section, as it fails to be so in the simplest case of discrete groups.

\begin{rem}
The image of the coboundary $b$ in the space of 1-cochains is cyclic. Indeed, the
cyclicity condition for 1-cochains is $F(g^{-1}) \sigma(g)= - F(g)$ and the
function $f_c(g) = c(1-\sigma(g))$ satisfies it, since $\sigma$ is a group morphism.
\end{rem}

\section{Conclusions and open problems}

In this short note we have demonstrated that the extension of the modular pairs
in involution and the Connes-Moscovici Hopf-cyclic cohomology is possible for commutative multiplier Hopf algebras with the cocyclic object based on the space 
of all bounded functions. We provide a counterexample showing that the restricted multiplier cannot be invariant under the cyclicity operator. The question, whether 
similar construction is possible for arbitrary regular multiplier Hopf algebras is still 
an open problem.

The definition of the modular pairs of involution for the algebra of functions over 
discrete groups and the related cohomology groups leads to the problem of relating
the presented cohomology theory to the already existing ones. In particular, it will be
interesting to compute the relevant cohomology for the examples of multiplier Hopf algebras as the one discussed in the example 3.3. We leave that for future work.


\end{document}